\documentclass[10pt]{article}
\usepackage[utf8]{inputenc}

\usepackage{a4wide,color,eucal,mathrsfs,dsfont}
\usepackage{amsmath,amsthm,amssymb,epsfig,bbm,graphicx}
\usepackage{parskip}
\usepackage{xcolor}
\usepackage{pdfsync}
\usepackage{times}
\usepackage{mathtools}

\mathchardef\emptyset="001F

\usepackage[normalem]{ulem}
\numberwithin{equation}{section}

\usepackage[colorlinks=true]{hyperref}
\hypersetup{linkcolor=blue, urlcolor=blue, citecolor=red}

\usepackage[shortlabels]{enumitem}
\setlist[enumerate]{font={\upshape}}
\setlist[enumerate,1]{label={(\alph*)}}

\usepackage[a4paper,top=3cm,bottom=2cm,left=3cm,right=3cm,marginparwidth=1.75cm]{geometry}

\usepackage[english]{babel,varioref}
\usepackage[T1]{fontenc}
\usepackage{url}
\usepackage{setspace}
\usepackage{verbatim}

\usepackage{bm}
\usepackage{bbm}

\usepackage{listings}

\newtheorem{thm}{Theorem}[section]
\newtheorem{theorem}[thm]{Theorem}
\newtheorem*{theoremstar}{Theorem}
\newtheorem{prop}[thm]{Proposition}

\newtheorem{lem}[thm]{Lemma}
\newtheorem{lemma}[thm]{Lemma}

\newtheorem{corollary}[thm]{Corollary}

\theoremstyle{definition}
    
\newtheorem{remark}[thm]{Remark}
\newtheorem{defn}[thm]{Definition}
\newtheorem{definition}[thm]{Definition}

\renewcommand{\H}{\mathbb{H}}

\newcommand{\N}{\mathbb{N}}

\newcommand{\R}{\mathbb{R}}
\newcommand{\X}{\mathbb{X}}
\newcommand{\Y}{\mathbb{Y}}

\newcommand{\cP}{{\ensuremath{\mathcal P}}}
\newcommand{\cT}{{\ensuremath{\mathcal T}}}

\newcommand{\rr}{{\boldsymbol{r}}}

\newcommand{\ggamma}{{\boldsymbol{\gamma}}}

\newcommand{\mmu}{{\boldsymbol{\mu}}}

\newcommand{\ssigma}{{\boldsymbol{\sigma}}}

\newcommand{\sfd}{{\mathsf d}}
\newcommand{\sfe}{{\mathsf e}}

\newcommand{\sfZ}{{\mathsf Z}}

\newcommand{\rmC}{{\mathrm C}}

\newcommand{\Kliminf}{K\kern-3pt-\kern-2pt\mathop{\rm lim\,inf}\limits}  
\newcommand{\supp}{\mathop{\rm supp}\nolimits}   
\renewcommand{\d}{{\mathrm d}}

\newcommand{\restr}[1]{\lower3pt\hbox{$|_{#1}$}}

\newcommand{\la}{{\langle}}                  
\newcommand{\ra}{{\rangle}}
\newcommand{\up}{\uparrow}
\newcommand{\eps}{\varepsilon}  
\newcommand{\nchi}{{\raise.3ex\hbox{$\chi$}}}
\newcommand{\weakto}{\rightharpoonup}
\newcommand{\GX}{X}
\newcommand{\PX}{{\X\times \Y}}
\newcommand{\PXtop}{{\X_s\times \Y_{\kern-1pt w}}} 
\newcommand{\PXname}{{\mathsf Z}}
\newcommand{\PXsw}{{\mathsf Z_{sw^{\kern-1pt*}}}}
\renewcommand{\X}{\mathsf X}
\renewcommand{\H}{\mathsf H}
\renewcommand{\Y}{\mathsf Y}
\newcommand{\id}{\mathrm{id}}
\newcommand{\wcP}[3]{\cP_{#1#2}^{sw}(#3)}
\newcommand{\testsw}[3]{\rmC^{sw}_{#1#2}(#3)}
\newcommand{\cPtHw}{\cP_2^w(\H)}

\newcommand{\Pcal}{\mathcal{P}}

\newcommand{\Fix}{\textnormal{Fix}}
\newcommand{\Pdue}{\mathcal{P}_2(X)}

\renewcommand{\d}{\mathrm d}

\DeclareMathOperator*{\argmin}{arg\,min}

\title{Weak topology and Opial property
  in Wasserstein spaces, with
  applications to Gradient Flows
  and Proximal Point Algorithms of
  geodesically convex functionals}
\author{Emanuele Naldi 
  \thanks{Universit\"atsplatz 2, 38106 Braunschweig, Germany.
    Email: \texttt{e.naldi@tu-braunschweig.de}}\\
  {\small \em Institut f\"ur Analysis und Algebra, TU Braunschweig.}
  \and
  Giuseppe Savar\'e
  \thanks{Via Roentgen 1, 20136 Milan, Italy.
    Email: \texttt{giuseppe.savare@unibocconi.it}}\\
  {\small \em Department of Decision Sciences,
  Bocconi University.}
}

\date{\small Dedicated to the memory of Claudio Baiocchi, outstanding
  mathematician and beloved mentor}

\begin{document}

\maketitle

\begin{abstract}
  In this paper we discuss how to define
  an appropriate notion of
  \emph{weak topology} in the Wasserstein space
  $(\cP_2(\H),W_2)$ of Borel probability measures
  with finite quadratic
  moment on a separable Hilbert space $\H$.
  
  We will show that such a topology inherits
  many features of the usual weak topology in Hilbert spaces, in
  particular the weak closedness of
  geodesically convex closed sets and
  the Opial property characterising
  weakly convergent sequences.
  
  We apply this notion to the approximation of fixed points for a
  non-expansive map in a weakly closed subset of $\cP_2(\H)$ and
  of minimizers of a lower semicontinuous and 
  geodesically convex functional $\phi:\cP_2(\H)\to(-\infty,+\infty]$
  attaining its minimum.
  In particular, we will show
  that every solution to the
  Wasserstein
  gradient
  flow of $\phi$
  weakly converge to a minimizer of $\phi$
  as the time goes to $+\infty$.
  Similarly, if $\phi$ is also convex along generalized geodesics,
  every sequence generated by the proximal point algorithm converges
  to
  a minimizer of $\phi$ with respect to the weak topology of
  $\cP_2(\H)$. 
\end{abstract}

\section{Introduction}

Opial proved in \cite{Opial67} that
weak convergence in a separable Hilbert space
$(\H,|\cdot|)$
admits a nice metric characterization.
\begin{theoremstar}[Opial]
  Let 
  $(x_n)_{n\in\mathbb{N}}$ be a sequence in $\H$
  weakly converging to $x \in \H$.
  Then
  \begin{equation}
    \label{eq:29}
    |y-x|^2+\liminf_{n\to\infty}|x_n-x|^2\le
    \liminf_{n\to\infty}|x_n-y|^2.
  \end{equation}
  In particular, for every $y\neq x$
  \begin{equation}\label{eq:1}
    \liminf_{n\to\infty}|x_n-x|<\liminf_{n\to\infty}|x_n-y|.
  \end{equation}
  \end{theoremstar}
  The \emph{proof} can be easily obtained by passing to the limit
  in the identity
  \begin{equation}
    \label{eq:30}
    |x_n-y|^2=|x_n-x|^2+|x-y|^2+2\langle x_n-x,x-y\rangle,
  \end{equation}
observing that
$\lim_{n\to\infty}\langle x_n-x,x-y\rangle=0$
by weak convergence.
It is worth noticing that \eqref{eq:1} shows that
the weak limit $x$ of a sequence $(x_n)_{n\in \N}$
is the unique strict minimizer of the function
\begin{equation}
  \label{eq:31}
  L(y):=\liminf_{n\to\infty}|x_n-y|.
\end{equation}
Opial property (further extended and studied in more general Banach
spaces, see e.g.~\cite{Prus92})
has many interesting applications. A first
one, which already appears as a relevant motivation in Opial's paper
\cite{Opial67}, is related to the approximation of
a fixed point of a non-expansive map $T:C\to C$
defined in a closed and convex subset of $\H$.
If the set of fixed points of $T$ is not empty
and $\lim_{n\to\infty} |T^{n+1} x-T^n x|=0$
for some $x\in C$, 
then the sequence of iterated maps
$(T^n x)_{n\in \N}$ weakly converges to
a fixed point $y\in C$ of $T$ as $n\to\infty$.

A second kind of applications concerns
the dynamic approximation of minimizers
of a convex and lower semicontinuous function
$\varphi:\H\to(-\infty,+\infty]$ as the asymptotic limit
of its gradient flow or of the
so called Proximal Point Algorithm.

More precisely, if $\argmin \varphi\neq\emptyset$, Bruck
\cite{Bruck75} proved that every locally Lipschitz curve 
$x:(0,+\infty)\to\H$ solving
the differential inclusion
(the gradient flow generated by $\varphi$)
\begin{equation}
  \label{eq:48}
  \frac\d{\d t}x(t)\in -\partial\varphi(x(t))\quad\text{a.e.~in $(0,\infty)$},
\end{equation}
weakly converges to a minimizer
$x_\infty$ of $\varphi$ as $t\up\infty$.
An analogous asymptotic behaviour is exhibited by the
solutions to the Proximal
Point Algorithm: selecting an initial datum $x_0\in \H$ and a time
step $\tau>0$, one considers
the sequence
$(x^k_\tau)_{k\in \N}$ which recursively solves the variational problem
\begin{equation}
  \label{eq:51}
  x^k_\tau\quad\text{minimizes}\quad
  y\mapsto \frac 1{2\tau}|y-x^{k-1}_\tau|^2+\varphi(y).
\end{equation}
A result of Martinet \cite{Martinet70,Martinet72}
(see also Rockafellar \cite{Rockafellar76})
shows that the sequence $(x^k_\tau)_{k\in
  \N}$ 
weakly converges to an element $x_\infty$
of $\argmin\varphi$.

The aim of the present paper is to study the extension of the Opial Lemma to the
metric space $(\cP_2(\H),W_2)$
of Borel probability measures 
on $\H$ endowed with the
Kantorovich-Rubinstein-Wasserstein distance $W_2$
and
to derive similar applications to fixed points and to convergence of
gradient flows
and proximal point algorithms.

Let us recall that
a Borel probability measure $\mu$ in $\H$ belongs to $\cP_2(\H)$ if
\begin{equation}
  \label{eq:52}
  \text{the quadratic moment }\int_\H |x|^2\,\d\mu(x)\quad\text{is finite}.
\end{equation}
The squared Wasserstein distance between $\mu_1,\mu_2\in \cP_2(\H)$
can then be defined as the solution of the Optimal Problem with
quadratic cost
\begin{equation}
  \label{eq:53}
  W_2^2(\mu_1,\mu_2):=
  \min\Big\{\int_{\H\times \H} |x_1-x_2|^2\,\d\mmu(x_1,x_2):
  \mmu\in \Gamma(\mu_1,\mu_2)\Big\},
\end{equation}
where $\Gamma(\mu_1,\mu_2)$ denotes the set of couplings between
$\mu_1$ and $\mu_2$, i.e.~the Borel probability measures
in $\H\times \H$ whose marginals are $\mu_1$ and $\mu_2$ respectively.
It turns out that $(\cP_2(\H),W_2)$ is a complete and separable metric
space, which contains an isometric copy of $\H$
given by the Dirac masses $\{\delta_x:x\in \H\}$
(see e.g.~\cite{Villani09,AGS08,Santambrogio15}).

Since the distance in $\cP_2(\H)$ 
cannot be derived by a norm, 
a first natural question concerns the appropriate definition
of a suitable weak topology in $\cP_2(\H)$, which enjoys
at least some of the most useful properties of weak convergence in
Hilbert spaces:
\begin{enumerate}[(a)]
\item bounded sequences admit weakly convergent subsequences,
\item the distance function from a given
  element is a weakly lower semicontinuous map,
\item the scalar product is sequentially continuous
  w.r.t.~strong/weak convergence of their factors,
\item weakly convergent sequences are bounded,
\item strongly closed convex sets are also weakly closed.
\end{enumerate}
A first approach, adopted in \cite{AGS08},
is to work with the Wasserstein distance
induced by a weaker metric on $\H$,
which metrizes the weak topology on bounded sets.
This provides a satisfactory answer to the first three
questions (a,b,c), at least for bounded sequences.

Here we rely on a different point of view, recalling that
the topology of $\cP_2(\H)$
could be equivalently characterized as the
initial topology induced by the family of real functions
$F_\zeta:\cP_2(\H)\to\R$ where
\begin{equation}
  \label{eq:54}
  F_\zeta:\mu\to \int_\H \zeta\,\d\mu,\quad
  \zeta\in \mathrm C(\H),\
  \sup_{x\in \H}\frac{\zeta(x)}{1+|x|^2}<\infty,
\end{equation}
i.e.~the coarsest topology such that makes
all the functions $F_\zeta$ in \eqref{eq:54} continuous.
We thus define the weak topology in $\cP_2(\H)$
as the initial topology $\sigma(\cP_2(\H),\rmC_2^w(\H))$ induced by
$F_\zeta$ as $\zeta$ varies in the set
\begin{equation}
  \label{eq:54bis}
  \rmC_2^w(\H):=\Big\{
  \zeta:\H\to\R\text{ is sequentially weakly continuous},
  \quad
  \lim_{|x|\to\infty} \frac{\zeta(x)}{1+|x|^2}=0\Big\},
\end{equation}
and we call $\cPtHw$ the corresponding topological space
$\big(\cP_2(\H), \sigma(\cP_2(\H),\rmC_2^w(\H))\big)$.
In this way, $\cPtHw$ inherits the weak$^*$ topology of a subset
of the dual of the Banach space $\rmC_2^w(\H)$ 
and we will show that it satisfies all the previous properties (a,\ldots,e).
In particular, we will prove
that every lower semicontinuous
geodesically convex function
$\phi:\cP_2(\H)\to(-\infty,+\infty]$ is also
sequentially lower semicontinuous w.r.t.~the
weak topology $\sigma(\cP_2(\H),\rmC_2^w(\H))$. As a byproduct,
for every $\mu_0\in \cP_2(\H)$ and $\tau>0$
the Proximal Point Algorithm in $\cP_2(\H)$
(also known as JKO \cite{JKO98} or
Minimizing Movement scheme \cite{AGS08})
\begin{equation}
  \label{eq:51bis}
  \mu^k_\tau\quad\text{minimizes}\quad
  \mu\mapsto \frac 1{2\tau}
  W_2^2(\mu,\mu^{k-1}_\tau)+\phi(\mu)
\end{equation}
has always a solution $(\mu^k_\tau)_{k\in \N}$.

We will then show that
the Opial property holds in $\cPtHw$,
with the same structure of
\eqref{eq:29}:
if $(\mu_n)_{n\in \N}$ is a sequence converging to
$\mu$ in $\cPtHw$ then
\begin{equation}
  \label{eq:55}
  W_2^2(\nu,\mu)+\liminf_{n\to\infty}
  W_2^2(\mu,\mu_n)\le
  \liminf_{n\to\infty}
  W_2^2(\nu,\mu_n)
  \quad\text{for every }\nu\in \cP_2(\H).
\end{equation}
Applications to the asymptotic
convergence of the gradient flows of a lower semicontinuous and
geodesically convex functional $\phi:\cP_2(\H)\to(-\infty,+\infty]$
can then be easily derived by the same strategy of \cite{Bruck75},
by using the metric characterization of a solution
$\mu:(0,\infty)\to D(\phi)$ of the gradient flow of $\phi$ in
$\cP_2(\H)$ 
in terms of Evolution Variational Inequalities
\cite{AGS08} (see \cite{Baiocchi89} for such a metric approach
to \eqref{eq:48} in
Hilbert spaces)
\begin{equation}
  \label{eq:63intro}
  \frac{1}{2}\frac{\d}{\d
    t}W_2^2(\mu_t,\sigma)\leq\phi(\sigma)-\phi(\mu_t)
  \quad\text{$\mathscr L^1$-a.e.~in $(0,\infty)$,\quad
    for every }\sigma\in D(\phi).
  \tag{EVI}
\end{equation}
Analogous results hold for the convergence of the Proximal Point Algorithm
(here we use the discrete estimates of \cite{AGS08} for
\eqref{eq:51bis} assuming convexity along generalized geodesics)
and for the approximation of the fixed point of a non-expansive and
asymptotically regular map $T$ defined in a weakly closed subset
of $\cP_2(\H)$.

\paragraph{\em Plan of the paper.}
We will collect in Section \ref{sec:preliminaries}
the main facts concerning optimal transport and 
Kantorovich-Rubinstein-Wasserstein distances;
we adopt a general topological framework, in order to include
convergence of Borel probability measures
with respect to non metrizable topologies as the weak topology
in a Hilbert space.

Section \ref{sec:sw} is devoted to the definition of
the weak topology of $\cPtHw$.
Here we adopt a more general viewpoint, considering probability
measures in product spaces $\mathsf Z=\mathsf X_s\times \mathsf Y_w$
where $\X_s$ is a Banach space with its strong topology and
$\mathsf Y_w$ is a reflexive Banach space endowed with its weak
topology under the general $p$-$q$ growth condition
\begin{displaymath}
  \int_{\X\times \Y}\Big(\|x\|_\X^p+\|y\|_\Y^q\Big)\,\d\mmu(x,y)<\infty
\end{displaymath}
and leading to the space $\cP_{pq}^{sw}(\X\times \Y)$.
This is quite useful to deal with
the integration of bilinear forms;
when $q=p=2$ and $\X=\Y=\H$ is an Hilbert space,
we recover the case which is particularly relevant for our further
applications.
The weak topology of $\cPtHw$ corresponds to the choice $\X=\{0\}$ and
$\Y=\H$.
As a byproduct, we will prove
the stability of optimal couplings w.r.t.~the convergence in
$\cP_{22}^{sw}(\H\times \H)$.

Section \ref{sec:wlsc} deals with the weak lower semicontinuity
in $\cPtHw$ of geodesically convex functionals.
The Opial property in $\cPtHw$ is discussed in Section \ref{sec:Opial}.

The last Section \ref{sec:applications}
contains the applications to the asymptotic behaviour of
gradient flows, of the Proximal Point Algorithm, and of
the iteration of non-expansive and asymptotically regular maps
in $\cP_2(\H)$.

\paragraph{\em Acknowledgments.}
G.S.~gratefully acknowledges the support of the Institute of Advanced
Study of the Technical University of Munich, of the MIUR-PRIN 2017
project
\emph{''Gradient flows, Optimal Transport and Metric Measure Structures''}
and of the IMATI-CNR, Pavia.
The authors also acknowledge the support of the Department of Mathematics ``F. Casorati" of the University of Pavia  during the preparation of a preliminary version of the paper.\\
E.N.~acknowledges that this project has received funding from the European Union’s Framework Programme for Research and Innovation Horizon 2020 (2014-2020) under the Marie Skłodowska-Curie Grant Agreement No. 861137
\includegraphics[width=.5cm, height=.33cm]{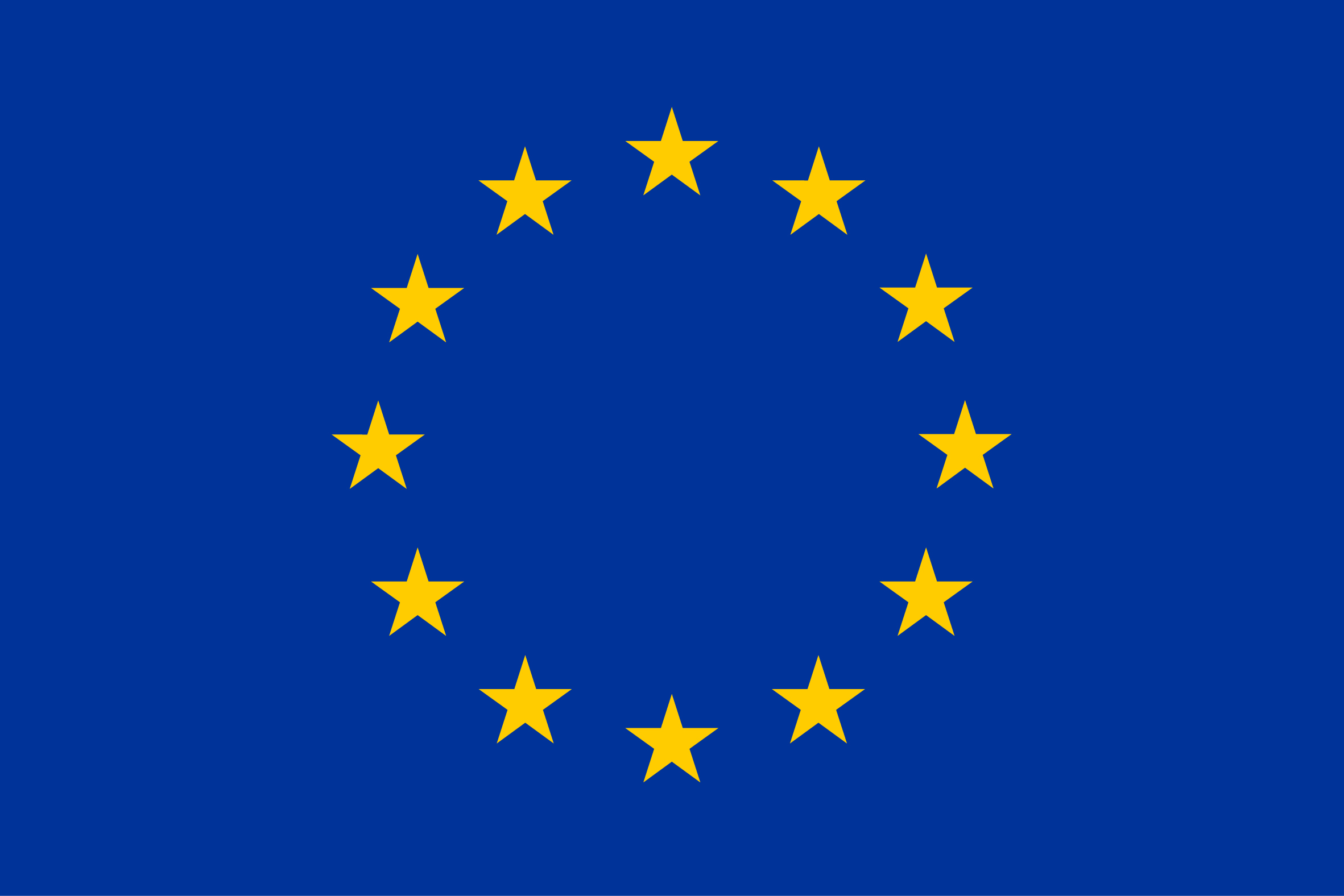}
\\
G.S.~graduated and began his research under the scientific direction of Claudio Baiocchi, to whom he is deeply
grateful and indebted for the generous and enlightening guidance, and his gentle humanity. Thanks to the
originality of his ideas, the depth of his vision, the superb clarity and broad scope of his teaching, and his
encouragement towards new and stimulating problems, Claudio Baiocchi had a fundamental influence on
the development of the "Pavia School". To remain in the context of the present paper, here we will limit ourselves
to quote his deep contribution to the study of evolution variational inequalities, in particular \cite{Baiocchi89}, which
had a profound impact on numerous developments and later on the metric theory of gradient flows.\\

\section{Preliminaries}
\label{sec:preliminaries}

\let\epsilon\eps
\let\smallsetminus\setminus
\subsection{Radon measures in completely regular topological spaces}

Let $(\GX,\cT)$ be a Hausdorff topological space.
We will denote by $\mathscr B(\GX)$ the Borel $\sigma$-algebra
in $\GX$ and by $\Pcal(\GX)$ the set of all \emph{Radon} probability
measures, i.e.~Borel probability measures satisfying
the
inner approximation property by compact sets 
\begin{equation}\label{eq:radonspaces}
  \forall B\in\mathscr{B}(X),\  \forall \eps>0 \qquad \exists K_{\epsilon}\subset B \text{ compact } \text{ such that } \mu(B\smallsetminus K_{\epsilon})\leq \epsilon.
\end{equation}
Recall that if $(\GX,\cT)$ is
Polish (i.e.~its topology is induced by a metric $\sfd$ such that $(X,\sfd)$ is
complete and separable) or 
it is Lusin (i.e.~$\GX$ admits a Polish topology $\cT'$ finer
than $\cT$) then 
every Borel probability measure on $\GX$ is also Radon.
We notice moreover that if $\mu$ is a Radon measure in a metric space $\GX$ then
its support $\operatorname{supp}(\mu)$ is separable.

We will mostly deal with Borel probability measures in separable
Hilbert/Banach spaces (or dual of separable Banach spaces),
possibly endowed with their weak/weak$^*$
topology. Clearly a separable Banach space is Polish;
the dual of a separable Banach space with its weak$^*$ topology
is a
Lusin space \cite[Theorem 7, page 112]{Schwartz73}.

In order to define the natural weak topology of $\cP(\GX)$ in such
general settings, let us recall
that 
a topological space $(\GX,\cT)$ is called \emph{completely regular} if it is
Hausdorff and for all closed set $F\subset X$ and for all $x_0\in
X\setminus F$ there exists $f\in \rmC_b(X)$ (the space of continuous
and bounded real functions defined in $(X,\cT)$) such that
$f(x_0)=0$ and $f\restr{F}\geq 1$. 
It is worth noticing that every metric space $(X,d)$
and every Hausdorff topological vector space (in particular every Banach space
endowed with the weak or the weak$^*$ topology) are completely
regular topological spaces.

\begin{defn}[Narrow topology in $\cP(X)$]
Let $(\GX,\cT)$ be a completely regular topological space. The
\emph{narrow topology}
in $\cP(\GX)$ is the coarsest topology in $\Pcal(X)$ such that all the
functionals
$\mu\mapsto\int_\GX f\,\d\mu$, $f\in \rmC_b(X)$, are continuous. In particular, we say that a sequence $(\mu_n)_{n\in\mathbb{N}}\subset \Pcal(X)$ is \emph{narrowly convergent} to $\mu\in\Pcal(X)$ if
\begin{equation}\label{eq:wconv}
    \lim_{n\to\infty}\int_Xf(x)\,\d\mu_n(x)=\int_Xf(x)\,\d\mu(x)\qquad \forall f\in \rmC_b(X).
    \end{equation}
\end{defn}
Narrow topology in $\cP(X)$ is also often called {\em weak topology};
since we will mostly deal with the case when $X$ is a Banach space endowed 
with its strong or weak topology
(and the corresponding topologies in $\cP(X)$), we will adopt the term {\em narrow} in order to avoid possible misunderstandings.

The Prokhorov Theorem (see \cite[III-59]{DM78} a proof)
provides an important criterium for relative
compactness
w.r.t.~the narrow topology.
\begin{defn}[Tightness]
A subset $\mathcal{K}\subset\Pcal(X)$ is \emph{tight} if $\forall\epsilon>0$ there exists a compact set $K_{\epsilon}\subset X$ such that $\mu(X\smallsetminus K_{\epsilon})<\epsilon$ for all $\mu\in\mathcal{K}$.
\end{defn}
\begin{thm}[Prokhorov]\label{thm:prok} 
Let $\mathcal{K}$ be a tight subset of $\Pcal(X)$, then $\mathcal{K}$ is relatively compact in $\Pcal(X)$ w.r.t. the narrow convergence. Conversely, if $X$ is a Polish space, then every relatively compact (w.r.t. the narrow convergence) subset of $\Pcal(X)$ is tight.
\end{thm}
The following result provides an useful integral condition for tightness.
\begin{prop}\label{prop:1}
$\mathcal{K}\subset\Pcal(X)$ is tight if and only if there exists a function $\varphi:X\to[0,+\infty)$, with compact sublevels, such that
\begin{equation}\label{eq:prop1}
\sup_{\mu\in\mathcal{K}}\int_X\varphi(x)\,\d\mu(x)\leq C <+\infty\end{equation}
\end{prop}


\subsection{Transport of measures}

\noindent Let $\GX_i$, $i=1,2$, be Hausdorff 
topological spaces, $\mu\in\Pcal(X_1)$ and $\rr:X_1\to X_2$ be a Lusin
$\mu$-measurable 
map (e.g.~a continuous map);
we denote by $\rr_{\sharp }\mu\in\Pcal(X_2)$ the \emph{push-forward} of $\mu$ through $\rr$, defined by
\begin{equation}
  \rr_{\sharp }\mu(B):=\mu(\rr^{-1}(B)) \qquad \forall B\in\mathscr{B}(X_2).
\end{equation}
A particularly important case is provided by the projection maps in
product spaces.
For an integer $N\geq 2$ and $i,j=1,...,N$, we denote by $\pi^i$ and
$\pi^{i,j}$ the projection operators defined on the product space
$\textbf{X}:=X_1\times...\times X_N$ and respectively defined by
\begin{equation}
    \pi^i:(x_1,...,x_N)\mapsto x_i\in X_i,\qquad \pi^{i,j}:(x_1,...,x_N)\mapsto (x_i,x_j)\in X_i\times X_j.
    \end{equation}
If $\mu^i\in\Pcal(X_i)$, $i=1,...,N$, the class of \emph{multiple plans} with marginals $\mu^i$ is defined by
\begin{equation}
    \Gamma(\mu^1,...,\mu^N):=\{\mmu\in\Pcal(\textbf{X}):\pi^i_{\sharp}\mu=\mu^i, i=1,...,N\}.
    \end{equation}
In the case $N=2$, $\mmu\in\Gamma(\mu^1,\mu^2)$ is also called a
\emph{transport plan}
or \emph{coupling} between $\mu^1$ and $\mu^2$.
\begin{remark}
  To every $\mu^1\in\Pcal(X_1)$ and every
  Lusin $\mu^1$-measurable map $\rr:X_1\to X_2$ and we can associate the transport plan
  \begin{equation}\label{eq:rinduce}
  \mmu:=(\id_{X_1}\times
  \rr)_{\sharp }\mu^1\in\Gamma(\mu^1,\rr_{\sharp }\mu^1),\quad
  \text{where $\id_{X_1}:X_1\to X_1$ is the identity map.}
    \end{equation}
If $\mmu$ is representable as in (\ref{eq:rinduce}) then we say that $\mmu$ is \emph{induced} by $\rr$.
\end{remark}

\bigskip

\noindent The following glueing lemma guarantees the existence of multiple plans with given marginals.

\begin{lem}\label{lem:givenmarginal}
Let $X_1,X_2,X_3$ be Lusin or metrizable spaces and let $\gamma^{12}\in\Pcal(X_1\times X_2)$, $\gamma^{13}\in\Pcal(X_1\times X_3)$ such that $\pi_{\sharp }^1\gamma^{12}=\pi_{\sharp }^1\gamma^{13}=\mu^1\in\Pcal(X_1)$. Then there exists $\ggamma\in\Pcal(X_1\times X_2 \times X_3)$ such that
\begin{equation}\label{eq:givenmarginal}
    \pi_{\sharp }^{12}\ggamma=\gamma^{12}, \qquad \pi_{\sharp }^{13}\ggamma=\gamma^{13}.
\end{equation}
\end{lem}
We denote by $\Gamma^1(\gamma^{1\,2},\gamma^{1\,3})$ the subset of plans $\mu\in\Pcal(X_1\times X_2\times X_3)$ satisfying (\ref{eq:givenmarginal}).



\subsection{Optimal Transport and Kantorovich-Rubinstein-Wasserstein spaces}
Let $(\GX,\mathsf d)$ be a metric space and let $p\in [1,+\infty)$.
We say that a Radon measure $\mu\in \cP(\GX)$ belongs to $\cP_p(\GX)$
if
\begin{equation}
  \label{eq:2}
  \int_{\GX} \mathsf d^p(x,x_o)\,\d\mu(x)<+\infty
  \quad\text{for some (and thus any) $x_o\in \GX$.}
\end{equation}
\begin{defn}
  The $L^p$-\emph{Kantorovich-Rubinstein-Wasserstein distance} $W_p$
  between two Radon probability measures $\mu^1,\mu^2\in\Pcal_p(\GX)$ is defined by
\[W_p^p(\mu^1,\mu^2) :=\min\Big\{\int \mathsf
  d^p(x_1,x_2)\,\d\mmu(x_1,x_2): \mmu\in\Gamma(\mu^1,\mu^2)\Big\}.\]
We denote by $\Gamma_o(\mu^1,\mu^2)\subset\Gamma(\mu^1,\mu^2)$
the convex and narrowly compact set of optimal plans where the minimum is attained, i.e.
\[
    \ggamma\in\Gamma_o(\mu^1,\mu^2) \iff \int \mathsf
    d^p(x_1,x_2)\,\d\ggamma(x_1,x_2)=
    W_p^p(\mu^1,\mu^2).
\]
\end{defn}
It is possible to prove that $\Gamma_o$ is not empty and
$(\cP_p(\GX),W_p)$ is a metric space.
It is easy to check that
a set $\mathcal{K}\subset \cP_p(\GX)$ is bounded
(i.e. there exists a measure $\nu\in\Pcal_p(\GX)$ such that
$\{W_p(\mu,\nu)\}_{\mu\in\mathcal{K}}$ is a bounded subset of $\R$)
if and only if
\begin{equation}
  \sup_{\mu\in\mathcal{K}}\int_{\GX}\sfd^p(x,x_o)\,\d\mu<+\infty
  \quad
  \text{for some (and thus any) point $x_o\in \GX$.}\label{eq:3}
\end{equation}
The following result shows the relationships
the narrow topology and the topology induced by the Wasserstein
distance $W_p$.
%
\begin{prop}\label{prop:carattconvWp}
  If $(X,\sfd)$ is separable (resp.~complete)
  then
  $(\Pcal_p(X),W_p)$ is a separable (resp.~complete)
  metric space. A set $\mathcal{K}\subset \Pcal_p(X)$ is relatively compact iff it has uniformly integrable $p$-moments and is tight. In particular, for a given sequence $(\mu_n)\subset\Pcal_p(X)$ we have
\begin{equation}\label{eq:convergenceequivalence}
    \lim_{n\to\infty}W_p(\mu_n,\mu)=0 \iff \begin{cases}\mu_n \text{ narrowly converge to }\mu,\\
    (\mu_n)\text{ has uniformly integrable $p$-moments.}
    \end{cases}
\end{equation}
\end{prop}


\begin{prop}[Stability of optimality and narrow lower semicontinuity]\label{prop:7.1.3}
Let $(\mu_n^1)$, $(\mu_n^2)\subset\Pcal_p(X)$ be two bounded sequences
narrowly converging to $\mu^1,\mu^2$ respectively, and let
$\mmu_n\in\Gamma_o(\mu_n^1,\mu_n^2)$ be a sequence of optimal plans.
Then $(\mmu_n)$ is narrowly relatively compact in $\Pcal(X\times X)$ and any narrow limit point $\mmu$ belongs to $\Gamma_o(\mu^1,\mu^2)$, with
\begin{equation}
    \begin{split}
        W_p^p(\mu^1,\mu^2) & = \int_{X^2}\sfd^p(x_1,x_2)\,\d\mmu(x_1,x_2)\\
        & \leq
        \liminf_{n\to\infty}\int_{X^2}\sfd^p(x_1,x_2)\,\d\mmu_n(x_1,x_2)
        =\liminf_{n\to\infty}W^p_p(\mu_n^1,\mu_n^2).
    \end{split}
\end{equation}
\end{prop}

\section{A strong-weak topology on measures in product spaces}
\label{sec:sw}
Let us consider
a
separable Banach space $\X$
endowed with the strong topology induced by
its norm $\|\cdot\|_\X$
(we will occasionally use the notation $\X_s$ when we want to emphasize
the choice of the strong topology)
and a reflexive and separable
Banach space $(\Y,\|\cdot\|_\Y)$.
We will denote by $\Y_w$
the space $\Y$ endowed with the weak topology $\sigma(\Y,\Y')$.

We are interested in Radon probability measures
in the topological space
$\PXtop$.
Since $\PXtop$ is endowed with the product topology of two Lusin and
completely regular spaces, it is a
completely regular Lusin space as well;
in particular Borel measures are Radon,
the set $\cP(\PXtop)$ coincides with $\cP(\X_s\times \Y_s)$,
and narrow convergence in
$\cP(\PXtop)$ is well defined.

Let us set $\PXname:=\X\times \Y$;
we want now to introduce a natural topology on the subset
\begin{equation}
  \label{eq:4}
  \cP_{pq}(\PXname):=\Big\{\mu\in \cP(\PXname):
  \int
  \Big(\|x\|_{\X}^p+\|y\|_{\Y}^q\Big)\,\d\mu(x,y)<+\infty\Big\}\quad
  p\in [1,+\infty),\ q\in (1,+\infty).
\end{equation}
In order to define such a topology, we consider the space $\testsw pq\PXname$ of test functions
$  \zeta:\PXname\to \R$ such that
\begin{gather}
  \label{eq:36}
  \zeta\text{ is sequentially continuous in $\PXtop$,}\\
  \label{eq:37}
  \forall\,\eps>0\ \exists\,A_\eps\ge 0: |\zeta(x,y)|\le
  A_\eps(1+\|x\|_\X^p)+\eps \|y\|_{\Y}^q\quad\text{for every }(x,y)\in \PX.
\end{gather}
We endow $\testsw pq\PXname$ with the norm
\begin{equation}
  \label{eq:32}
  \|\zeta\|_{\testsw pq\PXname}:=
  \sup_{(x,y)\in \PXname}\frac{|\zeta(x,y)|}{1+\|x\|_\X^p+\|y\|_{\Y}^q}.
\end{equation}
\begin{remark}
  When $\Y$ is finite dimensional, \eqref{eq:36} is equivalent to the
  continuity of $\zeta$.
  It is worth noticing that if 
  $p,q$ are conjugate exponents,
  any continuous and bilinear map
  $\beta: \PX\to \R$
  belongs to $\testsw pq\PX$. In fact,
  it is easy to check that $\beta$ is sequentially continuous
  in $\PXtop$ and its continuity yields the existence of
  a constant $L\ge 0$ such that
  \begin{displaymath}
    |\beta(x,y)|\le
    L\|x\|_\X\,\|y\|_{\Y}\quad\text{for every }x\in \X,\ y\in \Y,
  \end{displaymath}
  so that 
  \begin{displaymath}
    |\beta(x,y)|
    \le \frac {L^p}{p\eps^{p/q}}\|x\|_{\X}^p+\frac \eps q\|y\|_{\Y}^q
    \quad\text{for every }x\in \X,\ y\in \Y,\ \eps>0.
  \end{displaymath}
  This covers in particular the case when $\Y=\X$ is an Hilbert space
  and $\beta$ coincides
  with the scalar product $\la \cdot,\cdot\ra$ in $\X$.
\end{remark}
\begin{lemma}
  $(\testsw pq\PXname,\|\cdot\|_{\testsw
  pq\PXname})$ is a Banach space.
\end{lemma}
\begin{proof}
  It is obvious that $\|\cdot\|_{\testsw
    pq\PXname} $ is a norm, we can thus check the completeness.
    Let $\{\zeta_n\}_{n\in \N}$ be a Cauchy sequence in
    $\testsw pq
    \PXname$ and let $\zeta$ its pointwise limit, so that for every
    $\eta>0$ there exists $N_\eta\in \N$ such that 
    \begin{equation}
      \label{eq:38}
      |\zeta(x,y)-\zeta_n(x,y)|\le
      \eta(1+\|x\|_\X^p+\|y\|_{\Y}^q)\quad\text{for every }(x,y)\in
      \PX,\ n\ge N_\eta.
    \end{equation}
    If $(x_k, y_k)$ is a
    sequence converging to $(\bar x,\bar y)$ in $\PXtop$ as
    $k\to\infty$ we know that
    $$\|\bar x\|_\X^p+\|\bar y\|_{\Y}^q\le R:=
    \sup_k \|x_k\|_\X^p+\|y_k\|_{\Y}^q<\infty$$
    so that for every $n\ge N_\eta$
    \begin{displaymath}
      \limsup_{k\to\infty}|\zeta(\bar x,\bar y)-\zeta(x_k,y_k)|\le
      2(1+R) \eta+\limsup_{k\to\infty} |\zeta_n(\bar x,\bar y)-\zeta_n(x_k,y_k)|= 2(1+R) \eta.      
    \end{displaymath}
    Since $\eta$ is arbitrary, we conclude that $\zeta$ is
    sequentially continuous in $\PXname$.
    Let us eventually check that $\zeta$ satisfies \eqref{eq:37}.
    For a given $\eps>0$ we apply 
    \eqref{eq:38} with $\eta:=\eps/2$ and we pick up $n\ge
    N_\eta$. Since $\zeta_n$ belongs to $\testsw pq\PXname$, we find
    $B_\eps\ge0$ such that
    \begin{displaymath}
      |\zeta_n(x,y)|\le
      B_\eps(1+\|x\|_\X^p)+\eps/2 \|y\|_{\Y}^q\quad\text{for every }(x,y)\in \PX.
    \end{displaymath}
    Combining such inequality with \eqref{eq:38} we conclude that
    \begin{displaymath}
      |\zeta(x,y)|\le
      |\zeta(x,y)-\zeta_n(x,y)|+|\zeta_n(x,y)|\le
      (\eps+B_\eps)(1+\|x\|_\X^p)+\eps\|y\|_{\Y}^q.
    \end{displaymath}
  \end{proof}
  \renewcommand{\mmu}{\boldsymbol \mu}
  \begin{defn}[Topology of $\wcP pq\PX$]
    \label{def:topology}
        We endow {$\Pcal_{pq}( \PX)$} 
    with the initial topology 
    induced by the
    functions
    \begin{equation}
      \label{eq:39}
      \mmu\mapsto \int\zeta(x,y)\,\d\mmu(x,y),\quad
      \zeta\in \testsw pq{\PX},
    \end{equation}    
    and we call $\wcP pq \PX$ the topological space
    $\big(\cP_{pq}(\PX),\sigma\big(\Pcal_{pq}( \PX),\testsw pq{\PX}\big)\big).$
  \end{defn}
  It is obvious that whenever $r\ge p\lor q$
  the Wasserstein topology of $\cP_r(\X\times \Y)$ 
    (induced by the Wasserstein distance $W_r$ 
    generated by any product norm in the Banach space $\X\times \Y$)
  is finer than the
  topology of $\wcP pq\PX$ and
  the latter is finer than the narrow
  topology of $\cP(\PXtop)$.
  The 
  next proposition collects other important properties
  and justifies the
  interest of the $\wcP pq\PX$-topology.
  \begin{prop}
    \label{prop:finalmente}
    \begin{enumerate}
    \item If $(\mmu_\alpha)_{\alpha\in\mathbb A}\subset \wcP pq\PXname$ is
      a net indexed by the directed set $\mathbb A$ and $\mu\in
      \wcP pq\PXname$ satisfy
      \begin{enumerate}[(i)]
          \item[(i)] $\mmu_\alpha\to\mmu$ 
          narrowly in $\cP(\PXtop)$;
      \item[(ii)] $\displaystyle \lim_{\alpha\in \mathbb A}\int \|x\|_\X^p\,\d\mmu_\alpha=\int \|x\|_\X^p\,\d\mmu$; 
      \item[(iii)] $\displaystyle \sup_{\alpha\in\mathbb A} \int\|y\|_{\Y}^q\,\d\mmu_\alpha<\infty$,
      \end{enumerate}
      then $\mmu_\alpha\to\mmu$ in $\wcP pq\PXname$.
      The converse property holds for sequences: i.e.~if $\mathbb A=\N$
      and $\mmu_n\to\mu$ in $\wcP pq\PXname$ as $n\to\infty$, then
      properties (a), (b), (c) hold.
    \item For every compact set $\mathcal K\subset \cP_p(\X)$ and
      every constant $c<\infty$ the sets
      \begin{equation}
      \mathcal K_c:=\Big\{\mmu\in \wcP pq\PXname:\pi^1_\sharp\mmu\in
      \mathcal K,\quad
      \int \|y\|_{\Y}^q\,\d\mmu\le c\Big\}\label{eq:41}
    \end{equation}
    are compact and metrizable in
    $\wcP pq\PXname$ (in particular they are sequentially compact).
    \end{enumerate}
  \end{prop}
  \begin{proof}
    Let us consider the first claim and let 
    $(\mmu_\alpha)_{\alpha\in \mathbb A}$ in $\wcP pq\PXname$
    satisfy properties (i),
    (ii), (iii) with $S:=\sup_\alpha\int \|y\|_{\Y}^q\,\d\mmu_\alpha<\infty$.
    
    We first observe that $\pi^1_\sharp \mmu_\alpha\to \pi^1_\sharp \mmu$ in
    $\cP_p(\X)$.
    Let us now fix $\zeta\in \testsw pq\PXname$ and for every $\eps>0$ let
    $A_\eps$ as in \eqref{eq:37}. The function
    $\zeta_\eps(x,y):=\zeta(x,y)+A_\eps(1+\|x\|_\X^p)+2\eps \|y\|_{\Y}^q$
    is nonnegative and it is also lower semicontinuous w.r.t.~the
    $\PXtop$-topology: in fact, the sublevels 
    $X_{\eps,c}:=\{(x,y)\in \X\times \Y:\zeta_\eps(x,y)\le c\}$ of
    $\zeta_\eps$ are sequentially closed and contained in $\X\times\{y\in \X:\|y\|_{\Y}^q\le c/\eps\}$ which is a
    metrizable space, so that $X_{\eps,c}$ is closed in $\PXtop$.
    It follows that
    \begin{align*}
      \liminf_{\alpha\in \mathbb A}\int \zeta\,\d\mmu_\alpha
      &=\liminf_{\alpha\in \mathbb A}\int \zeta_\eps\,\d\mmu_\alpha-A_\eps\int
        \big(1+\|x\|_\X^p\big)\,\d\mmu_\alpha-2\eps\int \|y\|_\Y^q\Big)\,\d\mmu_\alpha
      \\&\ge \int \zeta_\eps\,\d\mmu-A_\eps\int
      \big(1+\|x\|_\X^p\big)\,\d\mmu-2\eps S
      \ge \int\zeta\,\d\mmu-2\eps S
    \end{align*}
    and, since $\eps>0$ is arbitrary, $      \liminf_{\alpha\in
      \mathbb A}\int
    \zeta\,\d\mmu_\alpha\ge \int\zeta\,\d\mmu$.
    Applying the same argument to $-\zeta$ we conclude that $\mmu_\alpha$
    converges to $\mmu$ in $\wcP pq\PX$.

    In order to prove the converse implication in the case of
    sequences, let us observe that 
    if $\mmu_n\to\mmu$ in
    $\wcP pq\PX$ then properties (i) and (ii) are obvious.
    Since $\testsw pq\PX$ is a Banach space and each measure $\mmu_n$
    induces a bounded linear functional $L_n$ on $\testsw pq\PX$,
    the principle of uniform
    boundedness implies that
    $S:=\sup_{n}\|L_n\|_{(\testsw pq\PXname)'}<\infty$, i.e.~
    \begin{equation}
      \label{eq:44}
      \int \zeta\,\d\mmu_n\le S\quad\text{for every }\zeta\in
      \testsw pq\PXname,\ |\zeta(x,y)|\le 1+\|x\|_\X^p+\|y\|_{\Y}^q.
    \end{equation}
    Let now $(\mathsf e_h)_{h\in \N}$ be a strongly dense subset
    of the unit ball of $\Y'$ (the dual of $\Y$, which is 
    separable as well) and let
    \begin{equation}
      \label{eq:45}
      \zeta_k(x,y):=\Big(\sup_{1\le h\le k}|\langle y,\mathsf
      e_h\rangle|\Big)^q\land k
    \end{equation}
    Clearly each $\zeta_k$ belongs to the unit ball of $\testsw pq\PXname$
    so that
    \begin{equation}
      \label{eq:46}
      \int \zeta_k(x,y)\,\d\mmu_n(x,y)\le S\quad\text{for every
      }k,n\in \N.
    \end{equation}
    Since $\zeta_k(x,y)\uparrow \|y\|_{\Y}^q$ as $k\to\infty$, Lebesgue
    Dominated Convergence Theorem yields $\int \|y\|_{\Y}^q\,\d\mmu_n\le S$
    for every $n\in \N$.

    \medskip\noindent
    (b)
    Since $\mathcal K$ is tight and $y\mapsto \|y\|_{\Y}^q$ has compact
    sublevel in $\Y_{w}$, the set $\mathcal K_c$ is tight in $\cP(\PXname)$
    and it is also closed, so that it is compact in $\cP(\PXname)$.
    Every net $(\mmu_\alpha)_{\alpha\in \mathbb A}$ in $\mathcal
    K_c$ has a subnet $(\mmu_{\alpha(\beta)})_{\beta\in
      \mathbb B}$ converging to $\mmu\in \mathcal K_c$ in $\cP(\PXname)$.
    Since $\pi^1_\sharp \mmu_\alpha$ is uniformly $p$-integrable
    we deduce $\lim_{\beta\in \mathbb
      B}\int\|x\|_\X^p\,\d\mmu_{\alpha(\beta)}=
    \int\|x\|_\X^p\,\d\mmu$. Applying the previous claim, we deduce that
    $\mmu_{\alpha(\beta)}\to \mmu$ w.r.t.~$\wcP pq\PX$.
    In order to prove the metrizability we observe that the bounded
    distance on $\Y$
    \begin{equation}
      \label{eq:1b}
      \sfd_\varpi (y_1,y_2):=\sum_{n=1}^\infty 2^{-n}(|\la y_1-y_2,\sfe_n\ra|\land 1)\quad
      \text{where $(\sfe_n)_{n\in \N}$ is dense in the unit ball of $\Y'$,}
    \end{equation}
    induces a coarser topology than $\sigma(\Y,\Y')$ in $\Y$, so that
    the $L_1$-Wasserstein distance associated to
    \begin{displaymath}
      \sfd((x_1,y_1),(x_2,y_2)):=|x_1-x_2|_\X+\sfd_\varpi(y_1,y_2)
    \end{displaymath}
    induces a coarser topology than the topology of $\wcP pq\PX$, which on the
    other hand coincides
    with the $\wcP pq\PX$-topology on the compact set $\mathcal
    K_c$. $\mathcal K_c$ is therefore metrizable.
  \end{proof}
  It is worth noticing that the topology of $\wcP pp \PX$ is
  strictly coarser than 
  the Wasserstein topology of $\cP_p(\X\times \Y)$ even when $\Y$ is finite
  dimensional. In fact, $\testsw pp\PX$ does not contain
  the function $(x,y)\mapsto \|y\|_{\Y}^p$, so that convergence of
  the $p$-moment w.r.t.~$y$ is not guaranteed.

  The previous construction is useful also in the case of a
  single space $\Y$ (we may think that $\X$ reduces to $\{0\}$).
  \begin{definition}[The topology of $\cP_q^w(\Y)$]
    \label{def:weak}
    Let $\Y$ be a reflexive and separable Banach space
    and $q\in (1,+\infty)$.
    \begin{enumerate}
    \item $\rmC_q^w(\Y)$ is the Banach space of sequentially weakly
      continuous (continuous, if $\Y$ is finite dimensional)
      functions $\zeta:\Y\to\R$ satisfying
      \begin{equation}
        \label{eq:49}
        \forall\,\eps>0\ \exists\,A_\eps\ge0:\quad
        |\zeta(y)|\le A_\eps+\eps\|y\|_{\Y}^q\quad\text{for every }y\in \Y,
      \end{equation}
      or, equivalently, 
      $\displaystyle
      \lim_{\|y\|_{\Y}\to\infty}\frac{\zeta(y)}{1+\|y\|_{\Y}^q}=0$,
      endowed with the norm
      \begin{equation}
        \label{eq:50}
        \|\zeta\|_{\rmC_q^w(\Y)}:=
        \sup_{y\in \Y}\frac{|\zeta(y)|}{1+\|y\|_{\Y}^q}.
      \end{equation}
    \item $\cP_q^w(\Y)$ is the topological space
      of measures in $\cP_q(\Y)$ endowed with
      the initial topology  $\sigma\big(\cP_q(\Y),\rmC_q^w(\Y)\big)$
      (or, equivalently, 
      the weak$^*$ topology of $\big(\rmC_q^w(\Y)\big)'$).
    \end{enumerate}
    \end{definition}
    The following result is an immediate consequence of
    Proposition \ref{prop:finalmente}.
    
    \begin{corollary}
      \label{cor:finalmente}
      Let $\Y$ be a reflexive and separable Banach space and $q\in (1,+\infty)$.
      \begin{enumerate}
      \item 
      The topology of $\cP_q^w(\Y)$ is finer than the narrow
      topology of $\cP(\Y_w)$;
      they coincide on bounded subsets $\mathcal K$ of $\cP_q(\Y),$ 
      i.e.~satisfying
      \begin{equation}
      \label{eq:bounded}
      \displaystyle \sup_{\mu\in \mathcal K}\int
        \|y\|_{\Y}^q\,\d\mu<\infty.
      \end{equation}
      \item a sequence $(\mu_n)_{n\in \N}$ converges to $\mu$ in
        $\cP_q^w(\Y)$ if and only if
      $$\displaystyle
      \text{$(\mu_n)_{n\in \N}$ converges narrowly in $\cP(\Y_{w})$ 
      \quad and} \quad \sup_{n\in
        \N}\int \|y\|_{\Y}^q\,\d\mu_n(y)<\infty.$$
    \item a set $\mathcal K\subset \cP_q(\Y)$ is relatively
      sequentially compact in $\cP_q^w(\Y)$ if and only if
      it satisfies
      \eqref{eq:bounded}.
      \item If a sequence $(\mmu_n)_{n\in \N}$ in $\wcP pq\PX$ converges to
            $\mmu$ in $\wcP pq\PX$ then
        $\pi^2_\sharp \mmu_n\to \pi^2_\sharp\mmu$ in $\cP_q^w(\Y)$
        (and $\pi^1_\sharp \mmu_n\to \pi^1_\sharp\mmu$ in
        $\cP_p(\X)$).
      \end{enumerate}
    \end{corollary}

\begin{remark}
  \label{rem:reflexive-case}
  All the definitions and results
  of this Section can be easily extended
  to the case when $\Y_w$ is replaced by
  the dual $\Y'_{w*}$ of a separable Banach space
  endowed with its weak$^*$-topology and
  we deal with the product
  $\X_s\times \Y_{w*}'$.
  We could therefore consider the spaces $\cP_{pq}^{sw*}(\X\times \Y')$
  and $\cP_q^{w*}(\Y')$. 
\end{remark}
Let us conclude this section with a useful application of
the weak topology of Definition \ref{def:topology} to the
stability of optimal plans in Hilbert spaces.
\begin{theorem}
  \label{thm:chefatica}
  Let $\H$ be a separable Hilbert space, let $(\mu^i_n)_{n\in \N}$
  be two sequences in $\cP_2(\H)$, $i=1,2$ with $\ggamma_n\in
  \Gamma_o(\mu^1_n,\mu^2_n)$, and let $\mu^1,\mu^2\in \cP_2(\H)$
  with $\ggamma\in \Gamma(\mu^1,\mu^2)$.
  If
  \begin{equation}
    \label{eq:10}
    (\mu^2_n)\text{ is tight in }\cP(\H_w)\quad\text{and}\quad
    \ggamma_n\to \ggamma\text{ narrowly in }\cP(\H_s\times \H_w)\text{ as
      $n\to\infty,$}
  \end{equation}
  then $\mu^1_n\to\mu^1$ narrowly in $\cP(\H)$, $\mu^2_n\to\mu^2$ narrowly in
  $\cP(\H_w)$
  and $\ggamma\in \Gamma_o(\mu^1,\mu^2).$
  In particular, any limit point $\ggamma$ of optimal plans $\ggamma_n$ in $\wcP
  22{\H\times \H}$ 
  is optimal as well.
\end{theorem}
\begin{proof}
  The statement concerning the convergence of $\mu^1_n$ and $\mu^2_n$
  is obvious.
  Since $\ggamma$ has finite quadratic moment, in order to check its
  optimality
  it is sufficient to prove that $\ggamma$ is concentrated
  on a cyclically monotone set,
  i.e.~there exists a Borel set $M\subset \H\times \H$
  such that $\ggamma(\H^2\setminus M)=0$ and
  for every $N\in \N$
  \begin{equation}
    \label{eq:11}
    (x_1^k,x_2^k)\in M,\ k=0,\cdots,N,\ 
    \text{with $(x_1^0,x_2^0)=(x_1^N,x_2^N)$}\quad\Rightarrow
    \quad
    \sum_{k=1}^N \langle x_1^k-x_1^{k-1},x_2^k\rangle\ge0.
  \end{equation}
  The standard idea, i.e.~using the convergence of the supports of
  $\ggamma_n$, should be adapted to the
  case of the (non-metrizable) weak topology of $\H$.
  We thus consider also the metric space $(\H_\varpi,\sfd_\varpi)$,
  whose metric has been defined by \eqref{eq:1b} (here $\Y=\H)$;
  we recall that the topology induced by $\sfd_\varpi$ coincides with
  the weak topology on every bounded subset of $\H$.

  Since $(\mu^1_n)_n$ is narrowly convergent in $\cP(\H)$
  it is tight, so that
  we can find a function
  $\psi_1:\H\to[0,+\infty]$
  with strongly compact sublevels such that
  $\int_\H \psi_1(x)\,\d\mu^1_n(x)\le S_1<\infty$
  for every $n\in \N$.
  Since $(\mu^2_n)$ is tight in $\cP(\H_w)$ we can find a function
  $\psi_2:\H\to[0,+\infty]$ with weakly compact sublevels such that
  $\int_\H \psi_2(x)\,\d\mu^2_n(x)\le S_2<\infty$ for every $n\in \N$.
  Let us set $\ssigma_n:=(\operatorname{Id}_{\H\times\H}\times
  \psi)_\sharp\ggamma_n\in \cP(\H^2\times [0,+\infty)).$
  We have that
  \begin{equation}
    \label{eq:12}
    \int\Big(\psi_1(x_1)+\psi_2(x_2)+|r|)\,\d\ssigma_n(x_1,x_2,r)
    \le S_1+2S_2
  \end{equation}
  so that the sequence $(\ssigma_n)_{n\in\N}$
  is tight in $\cP(\H\times \H_\varpi\times \R)$
  (recall \eqref{eq:1b}).
  Since $\H_s\times \H_\varpi\times \R$ is metrizable, 
  we can thus extract a subsequence (still denoted by $\ssigma_n$)
  converging to a limit plan $\ssigma\in \cP(\H\times \H\times \R)$
  such that $\pi^{12}_\sharp\ssigma=\ggamma$.

  Since $\ssigma$ is a Radon probability measure in $\H^2\times \R$,
  we can find
  an increasing sequence of compact sets
  $K_j\subset \supp(\ssigma)\subset \H^2\times \R$
  such that
  $\ssigma(\H^2\times \R\setminus \cup_j K_j)=0$.
  It follows that $\ggamma$ is concentrated
  on $M:=\cup_j M_j$
  where $M_j:=\pi^{12}\big(K_j)$ are compact sets.

  Let now $(x_1^k,x_2^k)$, $k=0,\cdots,N$, be points in $M$
  as in \eqref{eq:11}. There exists $j\in \N$ and
  points $r^k\ge0$ such that
  $(x_1^k,x_2^k,r^k)\in K_j$.
  Since
  $\ssigma_n$ is concentrated
  on $(\operatorname{Id}_{\H^2}\times \psi)(\supp(\ggamma_n))$,
  we can thus find
  a sequence $(x_{1,n}^k,x_{2,n}^k)\in \supp(\ggamma_n)$
  such that
  \begin{equation}
    \label{eq:13}
    x_{1,n}^k\to x_1^k\text{strongly in }\H,\
    \sfd_\varpi(x_{2,n}^k,x_2^k)\to 0,\
    \psi_2(x_{2,n}^k)\to r^k\text{ in $\R$ as $n\to\infty$}.
  \end{equation}
  Since $\psi_2$ has weakly compact sublevels,
  we deduce that $x_{2,n}^k\weakto x_2^k$
  as $n\to\infty$. Since $\ggamma_n$ is cyclically monotone,
  we know that
  \begin{equation}
    \label{eq:14}
    \sum_{k=1}^N \langle x_{1,n}^k-x_{1,n}^{k-1},x_{2,n}^k\rangle\ge0
    \quad\text{for every }n\in \N.
  \end{equation}
  We can then pass to the limit as $n\to\infty$ in \eqref{eq:14}
  and using the sequential continuity of the scalar product
  in $\H\times \H_w$ we obtain \eqref{eq:11}.  
\end{proof}

\section{Weak lower semicontinuity of geodesically convex functions in
$\cP_2(\H)$}
\label{sec:wlsc}
Let $(\H,|\cdot|)$ be a separable Hilbert space and
let $\varphi:\H\to \R\cup\{+\infty\}$ be a convex function.
It is well known that $\varphi$ is lower semicontinuous w.r.t.~the
strong topology of $\H$ if and only if it is lower semicontinuous
w.r.t.~the weak topology.
We want to extend this property to
geodesically convex functions
$\phi:\cP_2(\H)\to
\R\cup\{+\infty\}$, an important class
of functions introduced by McCann \cite{McCann97}.

Let us first recall that a (minimal, constant speed)
geodesic $(\mu_s)_{s\in [0,1]}$ in $\cP_2(\H)$ connecting two given measures
$\mu_0,\mu_1\in \cP_2(\H)$
is a Lipschitz curve satisfying
\begin{equation}
  \label{eq:21}
  W_2(\mu_s,\mu_t)=|t-s|W_2(\mu_0,\mu_1)\quad\text{for every }s,t\in [0,1].
\end{equation}
Equivalently, it is possible to prove (see
e.g.~\cite{AGS08})
that a curve $(\mu_s)_{s\in [0,1]}$ is a geodesic
if and only if there exists an optimal plan
$\mmu\in \Gamma_o(\mu_0,\mu_1)$
such that 
\begin{equation}
  \label{eq:20}
  \mu_s:=(\pi^{1\to2}_s)_\sharp \mmu,\quad
  \pi^{1\to2}_s(x_1,x_2):=(1-s)x_1+sx_2\quad x_1,x_2\in \H,\ s\in [0,1].
\end{equation}
\begin{definition}
  Let $\phi:\cP_2(\H)\to
  \R\cup\{+\infty\}$ be a function with proper domain
  $D(\phi):=\{\mu\in \cP_2(\H):\phi(\mu)<\infty\}\neq \emptyset.$
  $\phi$ is geodesically convex
  if every $\mu_0,\mu_1\in D(\phi)$
  can be connected by a geodesic $(\mu_s)_{s\in [0,1]}$
  in $\cP_2(\H)$ 
  along which
  \begin{equation}
    \label{eq:22}
    \phi(\mu_s)\le (1-s)\phi(\mu_0)+s\phi(\mu_1)\quad\text{for every
    }s\in [0,1].
  \end{equation}
  Equivalently,
  there exists $\mmu\in \Gamma_o(\mu_1,\mu_2)$ such that
  \begin{equation}
    \label{eq:23}
    \phi\big((\pi^{1\to2}_s)_\sharp\mmu\big)
    \le (1-s)\phi(\mu_0)+s\phi(\mu_1)\quad\text{for every
    }s\in [0,1].
  \end{equation}
\end{definition}

\begin{theorem}
  \label{thm:main-wlsc}
  Every
  lower semicontinuous and 
  geodesically convex function $\phi: \cP_2(\H)\to
  \R\cup\{+\infty\}$ is sequentially lower semicontinuous
  w.r.t.~the (weak) topology of $\cPtHw$:
  for every sequence $(\mu_n)_{n\in \N}$ and $\mu$ in $\cP_2(\H)$ we have
  \begin{equation}
    \label{eq:8}
    \mu_n\to \mu \text{ narrowly in }\cP(\H_w),\quad
    \sup_n\int |x|^2\,\d\mu_n<\infty\quad
    \Rightarrow\quad
    \liminf_{n\to\infty}\phi(\mu_n)\ge \phi(\mu).    
  \end{equation}
\end{theorem}
The proof of Theorem \ref{thm:main-wlsc}
(at the end of the present section) is
based on two preliminary results; 
the first one is an application to the Wasserstein space
of \cite[Theorem 2.10, 2.17]{Muratori-Savare20},
which hold in fact in an arbitrary complete metric space.
\begin{theorem}
\label{thm:MS}
Let $\phi: \cP_2(\H)\to
  \R\cup\{+\infty\}$ be a proper, lower semicontinuous, and 
  geodesically convex function.

  - $\phi$ is linearly bounded from below:
    there exists $\mu_o\in \cP_2(\H)$ and $\ell_o,\phi_o\in \R$ such that 
    \begin{equation}
      \label{eq:5}
      \phi(\mu)\ge \phi_o-\ell_o\,W_2(\mu,\mu_o)\quad
      \text{for every }\mu\in \cP_2(\H).
    \end{equation}
  - 
    For every $\mu\in D(\phi)$ and $\tau>0$
    there exists $\mu_\tau\in D(\phi)$ such that 
    \begin{align}
      \label{eq:6}
      \frac 1{2\tau}W_2^2(\mu_\tau,\mu)+\phi(\mu_\tau)&\le 
      \frac 1{2\tau}W_2^2(\mu,\nu)+\phi(\nu)+
                                                        W_2(\mu_\tau,\mu)W_2(\mu_\tau,\nu)
      \quad\text{for every }\nu\in D(\phi),\\
      \label{eq:6bis}
      \frac 1{2\tau}W_2^2(\mu_\tau,\mu)+\phi(\mu_\tau)&\le 
                                                        \phi(\mu)\\
      \label{eq:28}
      \lim_{\tau\downarrow0}W_2(\mu_\tau,\mu)&=0,\quad
      \lim_{\tau\downarrow0}\phi(\mu_\tau)=\phi(\mu).
    \end{align}
\end{theorem}
\begin{lemma}
  \label{le:step1}
  Let $\tau>0$, $\mu,\mu_\tau$ as in \eqref{eq:6} and \eqref{eq:6bis} of Theorem 
  \ref{thm:MS} and let $\mmu_\tau\in \Gamma_o(\mu_\tau,\mu)$. 
  For every $\nu\in D(\phi)$
  and $\ggamma_\tau\in \Gamma(\mmu_\tau,\nu)$
  such that $\pi^{13}_\sharp  \ggamma_\tau\in \Gamma_o(\mu_\tau,\mu)$
  and $\phi$
  satisfies the convexity inequality \eqref{eq:22}
  along $(\pi^{1\to3}_s)_\sharp\ggamma$,
  we have
  \begin{equation}
    \label{eq:7}
    \phi(\nu)-\phi(\mu_\tau)\ge
    \frac 1{\tau }\int
    \big\langle x_1-x_2,x_1-x_3\big\rangle \,\d\ggamma_\tau
    -W_2(\mu_\tau,\mu)W_2(\mu_\tau,\nu).
  \end{equation}
\end{lemma}
\begin{proof}
  Let $\ggamma_\tau\in \Gamma(\mmu_\tau,\nu)$
  as in the statement of the Lemma 
  and let $\nu_s:=(\pi^{1\to3}_s)_\sharp \ggamma_\tau$.
  Since $\phi$ 
  satisfies the convexity inequality \eqref{eq:22} along $(\nu_s)_{s\in [0,1]}$ 
  we have
  \begin{equation}
    \label{eq:24}
    \phi(\nu)-\phi(\mu_\tau)
    \ge
      \frac 1s\Big(\phi(\nu_s)-\phi(\mu_\tau)\Big).
    \end{equation}
    On the other hand, \eqref{eq:6}
    and the fact that $s^{-1}W_2(\nu_s,\mu_\tau)=W_2(\nu,\mu_\tau)$ yield
    \begin{equation}
      \label{eq:25}
    \frac 1s\Big(\phi(\nu_s)-\phi(\mu_\tau)\Big)
    \ge
    \frac 1{2\tau s}\Big(W_2^2(\mu_\tau,\mu)-
    \frac 1{2\tau}W_2^2(\mu,\nu_s)\Big)
      -W_2(\mu_\tau,\mu)W_2(\mu_\tau,\nu)
  \end{equation}
  Since $\pi^{12}_\sharp\ggamma_\tau$ is an optimal coupling between
  $\mu_\tau$ and $\mu$ and
  $(\pi^{1\to3}_s)_\sharp\ggamma_\tau=\nu_s$ we have
  \begin{displaymath}
    W_2^2(\mu_\tau,\mu)=
    \int
    |x_1-x_2|^2\,\d\ggamma_\tau,\quad
    W_2^2(\mu,\nu_s)\le \int|x_2-(1-s)x_1-sx_3|^2\,\d\ggamma_\tau
  \end{displaymath}
  so that \eqref{eq:25} yields
  \begin{equation}\label{eq:26}
     \frac 1s\Big(\phi(\nu_s)-\phi(\mu_\tau)\Big)\ge
    \frac 1{2\tau s}\int
    \Big(|x_1-x_2|^2-|x_2-(1-s)x_1-sx_3|^2\Big)\,\d\ggamma_\tau
    -W_2(\mu_\tau,\mu)W_2(\mu_\tau,\nu)
  \end{equation}
  Passing to the limit as $s\downarrow0$ in \eqref{eq:26}
  and recalling \eqref{eq:10} we eventually get \eqref{eq:7}.
\end{proof}
\begin{proof}[Proof of Theorem \ref{thm:main-wlsc}]
  It is not restrictive to assume that $\phi$ is proper
  and, possibly extracting a subsequence, that
  the limit $L:=\lim_{n\to\infty}\phi(\mu_n)$ exists and it is finite,
  where $\mu_n$ is a sequence as in \eqref{eq:8}.
  We set
  $S:=\sup_nW_2(\mu_n,\mu)$, which is finite since $(\mu_n)$ is
  bounded in $\cP_2(\H)$.
  
  For every $\tau>0$ let $\mu_\tau$
  be as in \eqref{eq:6} and \eqref{eq:6bis} of Theorem 
  \ref{thm:MS} and let $\ggamma_{\tau,n}\in \Gamma(\mmu_\tau,\mu_n)$
  as in the previous Lemma \ref{le:step1}.
  \eqref{eq:7} yields
    \begin{equation}
    \label{eq:7bis}
    \phi(\mu_n)\ge \phi(\mu_\tau)+
    \frac 1{\tau }\int
    \big\langle x_1-x_2,x_1-x_3\big\rangle \,\d\ggamma_{\tau,n}
    -W_2(\mu_\tau,\mu)\Big(W_2(\mu_\tau,\mu) +S\Big).
  \end{equation}
  Setting $\sfZ:=(\H^2)\times \H$, we can apply Proposition
  \ref{prop:finalmente}(b) to the sequence
  $(\ggamma_{\tau,n})_n$ obtaining a subsequence
  (still denoted by $\ggamma_{\tau,n}$) converging to
  a limit $\ggamma_\tau\in \Gamma(\mmu_\tau,\mu)$
  in $\wcP 22\sfZ$.
  Since 
  \begin{equation}
    \label{eq:9}
    \text{the map $A:\sfZ\to \R$ defined by  }
    A(x_1,x_2,x_3):=\big\langle x_1-x_2,x_1-x_3\big\rangle
    \quad\text{belongs to }\testsw 22\sfZ,
  \end{equation}
  by the very definition of the topology of $\wcP22\sfZ$ we get
  \begin{equation}
    \label{eq:16}
    \lim_{n\to\infty}
    \int\big\langle x_1-x_2,x_1-x_3\big\rangle \,\d\ggamma_{\tau,n}=
    \int\big\langle x_1-x_2,x_1-x_3\big\rangle \,\d\ggamma_{\tau}.
  \end{equation}
  On the other hand,
  by Theorem \ref{thm:chefatica}, $\pi^{2,3}_\sharp \ggamma_\tau$ is
  optimal, thus belongs to $\Gamma_o(\mu,\mu)$:
  it follows that it is concentrated on the
  subspace $\{(x_2,x_3)\in \H^2:x_2=x_3\}$ so that 
  \begin{equation}
    \label{eq:15}
    \int\big\langle x_1-x_2,x_1-x_3\big\rangle \,\d\ggamma_{\tau}=
    \int\big\langle x_1-x_2,x_1-x_2\big\rangle \,\d\ggamma_{\tau}=
    W_2^2(\mu_\tau,\mu).
  \end{equation}
  Combining \eqref{eq:7bis} with \eqref{eq:16}
  and \eqref{eq:15} we eventually get
  \begin{equation}
    \label{eq:17}
    L=\liminf_{n\to\infty}\phi(\mu_n)
    \ge \phi(\mu_\tau)+\frac {1-\tau}\tau W_2^2(\mu,\mu_\tau)-
    SW_2(\mu_\tau,\mu).
  \end{equation}
  Passing to the limit as $\tau\downarrow0$ in \eqref{eq:17} and
  applying \eqref{eq:28} we obtain $L\ge \phi(\mu)$.  
\end{proof}
We make explicit two interesting consequences of the previous result.
\begin{corollary}
  \label{cor:PPA}
  Let $\phi: \cP_2(\H)\to
  \R\cup\{+\infty\}$ be a proper, lower semicontinuous, and 
  geodesically convex function.
  Then for every $\mu\in \cP_2(\H)$ and every $\tau>0$
  there exists a solution $\mu_\tau\in D(\phi)$ of
  the problem
  \begin{equation}
    \label{eq:59}
    \mu_\tau
    \quad\text{minimizes}\quad
    \nu\mapsto \frac 1{2\tau}
    W_2^2(\nu,\mu)+\phi(\nu)\quad \nu\in D(\phi).
  \end{equation}
  In particular, the proximal point algorithm
  \eqref{eq:55}
  has always a solution for every initial measure $\mu_0\in \cP_2(\H)$.
\end{corollary}
\begin{corollary}
  Let $K$ be a geodesically convex set
  in $\cP_2(\H)$, i.e.~
  \begin{equation}
    \label{eq:57}
    \text{for every }\mu_0,\mu_1\in K\ \text{there exists }
    \mmu\in \Gamma_o(\mu_0,\mu_1):
    (\pi^{1\to2}_t)_\sharp\mmu\in K\quad
    \text{for every }t\in [0,1].
  \end{equation}
  If $K$ is closed in $\cP_2(\H)$ then it is also
  (weakly) sequentially closed in $\cP^w_2(H)$. In particular
  \begin{equation}
    \label{eq:58}
    \mu_n\in K,\ \sup_n\int|x|^2\,\d\mu_n<\infty,\
    \mu_n\to \mu\text{ in }\cP(\H_w)
    \text{ as }n\to\infty\quad
    \Longrightarrow\quad
    \mu\in K.
  \end{equation}
\end{corollary}
\section{Opial property}
\label{sec:Opial}
Having introduced a notion of weak convergence in $\cPtHw$ (see
Definition
\ref{def:weak}) 
which shares many properties of the weak topology in $\H$,
it is natural to investigate if the Opial property holds in 
$\cPtHw$.
This turns out to be true, as stated by the following result.

\begin{thm}[Opial property in $\cP_2(\H)$]\label{lem:opialHilbert}
  Let $(\mu_n)_{n\in\mathbb{N}}$ be a sequence
  weakly converging to $\mu$ in $\cPtHw$ according to
  Definition \ref{def:weak}. Then
  \begin{equation}
    \label{eq:33}
  W_2^2(\nu,\mu)+\liminf_{n\to\infty}W_2^2(\mu_n,\mu)\leq\liminf_{n\to\infty}W_2^2(\mu_n,\nu)\qquad\text{for
    every }\nu\in\mathcal{P}_2(\H).
\end{equation}
In particular,
\begin{displaymath}
  \liminf_{n\to\infty}W_2(\mu_n,\mu)<\liminf_{n\to\infty}
  W_2(\mu_n,\nu)\qquad\text{for every }\nu\in\mathcal{P}_2(\H)\text{ with } \nu\neq\mu.
\end{displaymath}
\end{thm}
\begin{remark}
  Notice that \eqref{eq:33} holds
  under the (seemingly) weaker assumption that
  $\mu_n\to\mu$ narrowly in $\cP(\H_w)$.
  In fact, \eqref{eq:33} trivially holds if
  $\liminf_{n\to\infty}W_2^2(\mu_n,\nu)=+\infty$.
  If the $\liminf$ is finite, then up to extracting a suitable
  subsequence
  we can always assume that $\mu_n$ is bounded in $\cP_2(\H)$
  so that narrow convergence in $\cP(\H_w)$ is equivalent to
  convergence in $\cPtHw$.
\end{remark}
\begin{proof}[Proof of Theorem \ref{lem:opialHilbert}]
  Let $\nu\in\Pcal_2(\H)$. Up to extracting a suitable
  subsequence it is not restrictive to assume that
  $\liminf_{n\to\infty} W_2^2(\mu_n,\nu)=
  \lim_{n\to\infty} W_2^2(\mu_n,\nu)<+\infty$ %
  
  By Lemma \ref{lem:givenmarginal} for all $n\in\mathbb{N}$ we can find
  $\ggamma_n\in\Gamma(\mu,\nu,\mu_n)$
  such that $\pi^{13}_\sharp \ggamma_n\in \Gamma_o(\mu,\mu_n)$
  and $\pi^{23}_\sharp \ggamma_n\in \Gamma_o(\nu,\mu_n)$.
  We have
\[
\begin{split}
  W_2^2(\mu_{n},\nu) & 
=\int_{X^3}|x_3-x_2|^2\,\d\ggamma_n(x_1,x_2,x_3)= \\
& = \int_{X^3}|x_3-x_1|^2\,\d\ggamma_n+\int_{X^3}|x_1-x_2|^2\,\d\ggamma_n+2\int_{X^3}\langle x_3-x_1,x_1-x_2\rangle\,\d\ggamma_n
\end{split}
\]
and therefore
\begin{equation}\label{eq:opialdimHilbert}
  W_2^2(\mu_{n},\nu)\geq W_2^2(\mu_{n},\mu)+W_2^2(\nu,\mu)+2\int_{X^3}\langle x_3-x_1,x_1-x_2\rangle\,\d\ggamma_n.
  \end{equation}
    Setting $\PXname:=(\H^2)_s\times \H_w$,
    we can apply Proposition \ref{prop:finalmente}(b) with $p=q=2$
    to the sequence $(\ggamma_n)_{n\in \N}$
    and find 
    a subsequence $(n_k)_{k\in \N}$ and
    $\ggamma\in\mathcal{P}_{2}(\H^3)$
    such that $\ggamma_{n_k}\to\ggamma$
    in $\Pcal_{22}^{sw}(\H^2 \times \H)$.
    By \eqref{eq:9} and the very definition of the  topology of
        $\cP_{22}^{sw}(\H^2\times \H)$ we can pass to the limit in
    \eqref{eq:opialdimHilbert}
    along the subsequence $n_k$ obtaining
    \begin{equation}\label{eq:opialdimHilbert2}
      \liminf_{n\to\infty}W_2^2(\mu_{n},\nu)\geq
      \liminf_{n\to\infty}W_2^2(\mu_{n},\mu)+W_2^2(\nu,\mu)+
      2\int_{X^3}\langle x_3-x_1,x_1-x_2\rangle\,\d\ggamma.
  \end{equation}
  On the other hand 
  $\pi^{13}_\sharp\ggamma_{n_k}\to \pi^{13}_\sharp\ggamma$ in
  $\cP_{22}^{sw}(\H\times\H)$;
  since
  $\pi^{13}_\sharp\ggamma_{n_k} \in \Gamma_o(\mu,\mu_{n_k})$,
  by Theorem \ref{thm:chefatica},
  $\pi^{13}_\sharp\ggamma\in \Gamma_o(\mu,\mu)$
  so that
\begin{equation}\label{eq:opialdim2Hilbert}
  \pi_{\sharp }^{1,3}\ggamma=(\mathrm{id}_\H\times
  \mathrm{id}_\H)_{\sharp }\mu,
\end{equation}
thus
  $\ggamma$ is concentrated on the subset
  $\{(x_1,x_2,x_3)\in \H^3:x_1=x_3\}$
  and therefore
  \begin{displaymath}
    \int_{X^3}\langle x_3-x_1,x_1-x_2\rangle\,\d\ggamma=0.
  \end{displaymath}
  Inserting this identity in \eqref{eq:opialdimHilbert2}
  we eventually get
  \eqref{eq:33}.
\end{proof}
\noindent In the simple finite dimensional case of
$\H=\R^d$ we obtain the following result.
\begin{corollary}[Opial property in $\Pcal_2(\mathbb{R}^d)$]
  \label{cor:opial}
  Let $(\mu_n)_{n\in\mathbb{N}}$ be a sequence in $\mathcal{P}_2(\mathbb{R}^d)$ and $\mu\in\mathcal{P}_2(\mathbb{R}^d)$. If $\mu_n \to \mu$ narrowly in $\Pcal(\R^d)$, then
\begin{equation}
     W_2^2(\nu,\mu)+\liminf_{n\to\infty}W_2^2(\mu_n,\mu)\leq\liminf_{n\to\infty}W_2^2(\mu_n,\nu)\qquad\text{for
       every }\nu\in\mathcal{P}_2(\mathbb{R}^d).
   \end{equation}
\end{corollary}

\section{Applications}
\label{sec:applications}
Let us first enucleate the technical core of many applications of
Opial Lemma. We state it in $\cP_2(\H)$, where $\H$ is a separable
Hilbert space as in the previous section.

\begin{lemma}
  \label{le:core}
  Let $\mathfrak T\subset (0,+\infty)$ be an unbounded set,
  let $\mu:\mathfrak T\to \cP_2(\H)$
  be a \emph{bounded} map
  and let $M$ be the set of limit points of $\mu$ in $\cPtHw$
  along diverging sequences:
  \begin{equation}
    \label{eq:27}
    M:=\{\nu\in \cP_2(H_w):\text{there exists an increasing
      sequence $(t_n)_{n\in \N}\subset \mathfrak T$}:
      \mu(t_n)\to\nu\text{ in }\cPtHw\}.
  \end{equation}
  If 
  \begin{equation}
    \label{eq:40}
    \text{for every $\nu\in M$
    the function $t\mapsto W_2(\mu(t),\nu)$ is decreasing in
    $\mathfrak T$}
  \end{equation}
  then there exists the limit
  $\displaystyle\lim_{t\to\infty\atop t\in \mathfrak T}\mu(t)$ in $\cPtHw$.
\end{lemma}
\begin{proof}
  Since the set $\mathcal{K}:=\{\mu(t): t\in \mathfrak T\}$ is bounded in
    $\cP_2(\H)$, it is contained in a compact and
  metrizable subset of $\cPtHw$.
  In particular
  $M$ is not empty and
  every $(\mu_{t_n})_{n\in\N}$
  corresponding to a diverging sequence $t_n\up\infty$, $t_n\in
  \mathfrak T$,
  has a convergence subsequence in $\cPtHw$.
  For every $\nu\in M$ we set
  \begin{equation}
    \label{eq:43}
    L(\nu):=\inf_{t\in \mathfrak T} W_2^2(\mu(t),\nu)=
    \lim_{t\to\infty\atop t\in \mathfrak T}
    W_2^2(\mu(t),\nu).
  \end{equation}
  In order to prove the existence of the limit
  it is therefore sufficient to show that 
  if $s_n,t_n\uparrow +\infty$ as $n\to\infty$
  are diverging sequences in $\mathfrak T$ such that
  the corresponding sequences $(\mu(s_n))_{n\in \N}$ and
  $(\mu(t_n))_{n\in \N}$ respectively converge to 
  $\nu'$ and $\nu''$ in $\cPtHw$ then 
  $\nu'=\nu''$.

  Since $\nu',\nu''\in M$,
  \eqref{eq:40} yields
  \begin{align*}
    L(\nu')&=      
        \liminf_{n\to\infty}W_2^2(\mu(t_n),\nu')=
        \liminf_{n\to\infty}W_2^2(\mu(s_n),\nu'),
    \\
    L(\nu'')&=
         \liminf_{n\to\infty}W_2^2(\mu(s_n),\nu'')
         =       \liminf_{n\to\infty}W_2^2(\mu(t_n),\nu'').
  \end{align*}
  Applying \eqref{eq:33} of Theorem \ref{lem:opialHilbert}
  first to the sequence $(\mu(s_n))_{n\in \N}$
  and then to the sequence $(\mu(t_n))_{n\in \N}$ we eventually get
\begin{align*}
  W_2^2(\nu',\nu'')+L(\nu')&\le L(\nu'')\\
  W_2^2(\nu'',\nu')+L(\nu'')&\le L(\nu')
\end{align*}
which imply $W_2(\nu',\nu'')=0$.
\end{proof}

\subsection{Convergence of Gradient Flows}
\label{section:3.1}

Let $\H$ be a separable Hilbert space
and let $\phi:\cP_2(\H)\to (-\infty,+\infty]$ be a proper, lower
semicontinuous
and geodesically convex functional such that
$\argmin\phi$ is not empty.

We want to study the asymptotic behaviour of
the gradient flows of $\phi$.

\begin{definition}
  A locally Lipschitz curve $\mu:(0,\infty)\to\cP_2(\H)$
  is a gradient flow of $\phi$ in the {\rm EVI} sense if
  it satisfies
  \begin{equation}
    \label{eq:34}
    \frac{1}{2}\frac{\d}{\d
      t}W_2^2(\mu_t,\sigma)\leq\phi(\sigma)-\phi(\mu_t)
    \quad\text{$\mathscr L^1$-a.e.~in $(0,\infty)$,\quad
      for every }\sigma\in D(\phi).
    \tag{EVI}
  \end{equation}
\end{definition}

\begin{thm}\label{thm:centrale2} Let $\phi:\mathcal{P}_2(\H)\to
  (-\infty,+\infty]$
  be a proper, l.s.c.~and geodesically convex functional and
  let $\mu:(0,+\infty)\to\mathcal{P}_2(\H)$ be a Gradient Flow in the
  {\rm EVI} sense. 
  Then $\argmin \phi\neq\emptyset$
  if and only if the curve $(\mu_t)_{t\ge 1}$ is bounded in
  $\cP_2(H)$;
  in this case 
  there exists $\mu\in \argmin\phi$ such that
  $\mu_t\to \mu$ in $\cPtHw$ as $t\to+\infty$.
\end{thm}
\begin{proof} Let
  us first remark that if $\nu$ is a minimizer of $\phi$ then
  \eqref{eq:34} yields
  \begin{equation}
    \label{eq:35}
    t\mapsto W_2(\mu_t,\nu)\quad\text{is decreasing in }(0,+\infty).
  \end{equation}
  In particular if $\argmin\phi\neq\emptyset$
  the set
  $\mathcal{K}:=\{\mu_t:t\ge 1\}$ is also bounded in
  $\cP_2(\H)$ 

  Let us now show that if $\mathcal K$ is bounded
  and $\mu$ is
    a limit point of $(\mu_t)_{t>0}$ along a
    diverging sequence
    $t_n\uparrow\infty$ then $\mu$ is a minimizer of $\phi$
    (this shows in particular that
    $\argmin \phi\neq\emptyset$).
  %
  %
    
  We integrate the \eqref{eq:34} equation form $1$ to $t>1$
  and we divide both sides by $t-1$, obtaining
  \[\frac{1}{2(t-1)}W_2^2(\mu_t,\sigma)+\frac{1}{(t-1)}\int_{1}^t\phi(\mu_r)
    \,\d r\leq\frac{1}{2(t-1)}W_2^2(\mu_1,\sigma)+\phi(\sigma).\]
  Since $t\mapsto\phi(\mu_t)$ is not increasing (see \cite[Theorem 3.3]{Muratori-Savare20})
  we have
  \[\phi(\mu_t)+\frac{1}{2(t-1)}W_2^2(\mu_t,\sigma)\leq\frac{1}{2(t-1)}W_2^2(\mu_1,\sigma)+\phi(\sigma)\]
which yields
\[\limsup_{t\to\infty}\phi(\mu_t)\leq\phi(\sigma)\quad\text{for every
  }\sigma\in D(\phi)\]
since $W_2^2(\mu_t,\sigma)$ is bounded.
By the lower semicontinuity of $\phi$ in $\cPtHw$ we have
\[\phi(\mu)\leq\liminf_{k\to\infty}\phi(\mu_{t_k})
  \leq\limsup_{t\to\infty}\phi(\mu_t)\leq\phi(\sigma)
  \quad \text{for all }\sigma,\]
so that $\mu$ is a minimizer of $\phi$.

The previous argument shows that the set
$M$ defined as in \eqref{eq:27} (choosing
$\mathfrak T:=[1,\infty)$)
is contained in $\argmin\phi$, so that
it satisfies \eqref{eq:40} thanks to \eqref{eq:35}.
Applying Lemma \ref{le:core} we conclude that
the curve $\mu_t$ converges to a limit $\mu\in M$ as $t\to\infty$
in $\cPtHw$;
in particular, $\mu$ is a minimizer of $\phi$.
%
%
%
\end{proof}





\subsection{Weak convergence of
  the Proximal Point Algorithm}\label{section:3.2}
Under the same assumptions of the previous Section
\ref{section:3.1},
we want to study
the asymptotic properties of
the \emph{Proximal Point Algorithm}
\eqref{eq:51bis}.
First we define the (multivalued) operator
\begin{equation}\label{eq:PPAminimize}
  J_{\tau}(\mu)=
  \argmin_{\nu\in\Pcal_2(\H)}\big\{\Phi_{\tau}(\mu,\nu)\big\},
  \quad
  \Phi_\tau(\mu,\nu):=
  \phi(\nu)+\frac{1}{2\tau}W_2^2(\nu,\mu).
\end{equation}
Thanks to
Corollary \ref{cor:PPA},
for every choice of $\mu_0\in\Pdue$ and $\tau>0$,
the PPA algorithm generates a sequence of points $(\mu_\tau^{k})_{k\in\mathbb{N}}$ which solves
\begin{equation}\label{eq:PPAa}
    \begin{cases} \mu_\tau^0=\mu_0\\ \mu_\tau^{k+1}\in J_{\tau}(\mu_\tau^{k}) & k=1,2...
    \end{cases}
\end{equation}
%
As for the study of the convergence of the Minimizing Movement method
in \cite{AGS08}, the crucial property to
study the asymptotic behaviour of the PPA scheme relies on the notion
of convexity along generalized geodesics.

\begin{defn}[Convexity along generalized geodesics]
  \label{def:cgen}
$\phi:\Pcal_2(\H)\to(-\infty,+\infty]$ is called \emph{convex along
  generalized geodesics} if for every choice of $\nu,\mu_0,\mu_1$ in
$D(\phi)$ there exists a
plan $\ggamma\in\Gamma(\nu,\mu_0,\mu_1)$
with $\pi_{\sharp }^{1,2}\ggamma\in\Gamma_o(\nu,\mu_0)$,
$\pi_{\sharp }^{1,3}\ggamma\in\Gamma_o(\nu,\mu_1)$, such that
\[\phi(\mu_t^{2\to3})\leq (1-t)\phi(\mu_0)+t\phi(\mu_1)\qquad \forall t\in[0,1].\]
\end{defn}

\begin{remark}
  The curve
$\mu_t^{2\to3}$ defined by 
\[\mu_t^{2\to3}=(\pi_t^{2\to3})_{\sharp }\ggamma\qquad t\in[0,1]\]
where $\ggamma$ satisfies the conditions of Definition
\ref{def:cgen} is called a \emph{generalized geodesic}
connecting $\mu_0$ to $\mu_1$ with reference measure $\nu$.
If $\phi:\Pcal_2(\H)\to(-\infty,+\infty]$ is a functional which is convex along generalized geodesics, then for every choice of $\nu,\mu_0,\mu_1$ in $D(\phi)$
the map $t\mapsto \Phi_{\tau}(\nu,\mu_t^{2\to3})$ satisfies the inequality
\begin{equation}
  \Phi_{\tau}(\nu,\mu_t^{2\to3})\leq (1-t)\Phi_\tau(\nu,\mu_0)+
  t\Phi_\tau(\nu,\mu_1)-\frac{1}{2\tau}t(1-t)W_2^2(\mu_0,\mu_1).
\end{equation}
Convexity along generalized geodesics implies convexity along geodesics
(see \cite[Lemma 9.2.7]{AGS08} for a proof).
\end{remark}
\begin{thm}\label{thm:perPPA}
  Let us suppose that $\phi:\cP_2(\H)\to(-\infty,+\infty]$
  is proper, lower semicontinuous, and convex along generalized
  geodesics, $\mu_0\in \overline{D(\phi)}$, and 
  $\tau>0$.
\begin{enumerate}[(i)]
    \item The PPA algorithm \eqref{eq:PPAa} has a unique solution
      $(\mu^k_\tau)_{k\in \N}$.
    \item For each $\nu\in D(\phi)$ and $k\ge 1$ we have
      \begin{equation}
        \label{eq:perconvPPA}
      \frac{1}{2\tau}W_2^2(\mu^k_\tau,\nu)-
      \frac{1}{2\tau}W_2^2(\mu^{k-1}_\tau,\nu)\leq
      \phi(\nu)-\phi(\mu^k_\tau)-
      \frac{1}{2\tau}W_2^2(\mu_\tau^{k},\mu_\tau^{k-1}).
    \end{equation}
  \item
    In particular for every $k\ge1$ we have
    \begin{equation}
      \label{eq:60}
      \frac1{\tau}W_2^2(\mu^k_\tau,\mu^{k-1}_\tau)
      +\phi(\mu^k_\tau)\le
      \phi(\mu^{k-1}_\tau)
    \end{equation}
    and the sequence $k\mapsto \phi(\mu^k_\tau)$ is not increasing.
\end{enumerate}
\begin{proof}
See \cite[Theorem 4.1.3]{AGS08}.
\end{proof}
\end{thm}

%
\begin{thm}[Convergence to a minimum]\label{thm:convergencetoamin} Let
  $\phi:\cP_2(\H)\to(-\infty,+\infty]$ be
  proper, lower semicontinuous and convex along generalized geodesics
  and let $(\mu^k_\tau)_{k\in \N}$ be a solution to the PPA algorithm
  \eqref{eq:PPAa}.
  Then $\argmin\phi\neq \emptyset$ if and only if
  $(\mu^k_\tau)_{k\in \N}$ is bounded in $\cP_2(\H)$.
  If this is the case,
  there exists the limit $\mu:=\lim_{k\to\infty}\mu^k_\tau$
  in $\cPtHw$ and $\mu\in \argmin \phi$.
  \end{thm}
  \begin{proof}
    If
    $\nu\in\argmin \phi$, (\ref{eq:perconvPPA}) yields
\[\frac{1}{2\tau}W_2^2(\mu_\tau^k,\nu)-\frac{1}{2\tau}W_2^2(\mu_\tau^{k-1},\nu)\leq
  0
\quad\text{for every }k\ge1\]
so that
\begin{equation}
  \label{eq:62}
  \text{the sequence}\quad
  k\mapsto W_2(\mu_\tau^k,\nu)\quad\text{is decreasing};
\end{equation}
in particular the set $\mathcal K:=\{\mu^k_\tau:k\in \N\}$
is bounded.

Conversely, if $\mathcal K$ is bounded and
$\mu$ is the weak limit of a subsequence $\mu^{k(n)}_\tau$
in $\cPtHw$ as $n\to\infty$, we want to prove
that $\mu\in \argmin\phi$.

Notice that since $k\mapsto \phi(\mu^k_\tau)$ is not increasing
and $\phi$ is sequentially lower semicontinuous in $\cPtHw$
we have $\mu\in D(\phi)$.

By summing both sides of (\ref{eq:perconvPPA}) from $1$ to $K$
and dividing by $K$, we obtain for every $\nu\in D(\phi)$ 
\[\frac{1}{2\tau}\frac{1}{K}\sum_{k=1}^{K}\big(W_2^2(\mu_\tau^k,\nu)-W_2^2(\mu_\tau^{k-1},\nu)\big)\leq
  \phi(\nu)-\frac{1}{K}\sum_{k=1}^K\phi(\mu_\tau^k)
\]
and therefore
\[\frac{1}{K}\sum_{k=1}^K\phi(\mu_\tau^k)+\frac{1}{2\tau}\frac{1}{K}W_2^2(\mu_\tau^K,\nu)\leq \phi(\nu) + \frac{1}{2\tau}\frac{1}{K} W_2^2(\mu_\tau^{0},\nu).\]
Since $k\to\phi(\mu_\tau^k)$ is not increasing by
Theorem \ref{thm:perPPA}(iii), we have
\[\phi(\mu_\tau^K)+\frac{1}{2\tau}\frac{1}{K}W_2^2(\mu_\tau^K,\nu)\leq \phi(\nu) + \frac{1}{2\tau}\frac{1}{K} W_2^2(\mu_\tau^{0},\nu).\]
Taking the $\limsup$ of this inequality
as $K\to\infty$ and using
the fact that $K\mapsto W_2(\mu_\tau^K,\nu)$ is bounded
and $\phi$ is sequentially lower semicontinuous w.r.t.~$\cPtHw$ convergence,
we obtain
\[\phi(\mu)\leq\liminf_{n\to\infty}\phi(\mu_\tau^{k(n)})
\leq\limsup_{K\to\infty}\phi(\mu_\tau^{K})\leq \phi(\nu)
\quad\text{for every }\nu\in D(\phi)\]
so $\mu\in\argmin\phi$.

The above argument shows that
the set $M$ of the limit points
of $(\mu^k_\tau)_{k\in \N}$
in $\cPtHw$ (defined as in \eqref{eq:27}
with $\mathfrak T:=\N$)
is included in $\argmin\phi$,
so that it satisfies condition \eqref{eq:40} thanks to \eqref{eq:62}.
We can eventually apply Lemma \ref{le:core} and
obtain the weak convergence of $(\mu^k_{\tau})_{k\in \N}$
in $\cPtHw$ as $k\to\infty$.
\end{proof}



\subsection{Fixed points of non-expansive and
  asymptotically regular maps}
\label{section:3.3}
We conclude this section by 
proving the weak convergence of
the iteration of a non-expansive and asymptotically regular map
$T:A\to A$ defined in a (weakly) closed subset $A$ of $\cP_2^w(\H)$.
The proof is a simple extension
to the Wasserstein setting of the original argument of Opial \cite{Opial67}.
\begin{defn} Let $A\subset\Pcal_2(\H)$; a map $T:A\to A$ is called \emph{non-expansive} if
\[W_2(T(\mu),T(\nu))\leq W_2(\mu,\nu)\qquad \text{for all $\mu,\nu\in
    A$}\]
$T$ is called \emph{asymptotically regular} if 
\[\lim_{k\to\infty}W_2(T^{k+1}(\mu),T^k(\mu))=0\quad\text{for every } \mu\in
  A.\]
\end{defn}
\begin{thm}
  Let $A$ be a (weakly) closed subset of $\Pcal_2^w(\H)$,
  let $T:A\to A$ be a non-expansive and asymptotically regular map,
  and let $\mu_k:=T^k(\mu)$, $k\in \N$, for some $\mu\in A$.
  
  Then $T$ has a fixed point
  if and only if $(\mu_k)$ is bounded in $\cP_2(\H)$;
  in this case it 
  converges in $\cP_2^w(\H)$ to a fixed point $\mu$ of $T$ as $k\to\infty$.
  \begin{proof}
    Let us denote by $\Fix(T)$ the set of fixed points
    of $T$.
    We first observe that 
    \begin{equation}
      \label{eq:18}
      \text{for every }\nu\in \Fix(T)\quad
      \text{the sequence }
      k\mapsto W_2(\mu_k,\nu)
      \quad\text{is not increasing}.
    \end{equation}
    In fact 
    \[W_2(\mu_{k+1},\nu)=W_2(T(\mu_{k}),T(\nu))\leq W_2(\mu_k,\nu)
      \qquad \text{for every }k\in\mathbb{N},\]
    since $T$ is non-expansive.
    In particular, if $\Fix(T)\neq\emptyset$ then
    the sequence $(\mu_k)_{k\in \N}$ is bounded.

    Let us now suppose that
    $\mathcal K:=\{\mu_k:k\in \N\}$ is bounded in $\cPtHw$
    and let us show that
    if $\mu$ is the weak limit of $\mu_{k(n)}$
    as $n\to\infty$ along an increasing subsequence $n\mapsto k(n)$,
    then $\mu\in \Fix(T)$.
    By Opial Lemma we have
    \begin{equation}
      \label{eq:56}
      W_2^2(\mu,T(\mu))+
      \liminf_{n\to\infty}W_2^2(\mu,\mu_n)
      \le
        \liminf_{n\to\infty}W_2^2(T(\mu),\mu_n).
    \end{equation}
    Since $\lim_{n\to\infty}W_2(\mu_n,T(\mu_n))=0$
    by the asymptotic regularity of $T$, we obtain
    \begin{displaymath}
      \liminf_{n\to\infty}W_2^2(T(\mu),\mu_n)=
      \liminf_{n\to\infty}W_2^2(T(\mu),T(\mu_n))\le
      \liminf_{n\to\infty}W_2^2(\mu,\mu_n).
    \end{displaymath}
    Combining this inequality with \eqref{eq:56} we obtain
    $W_2(\mu,T(\mu))=0$, i.e.~$\mu\in \Fix(T)$.

    Still assuming that $(\mu_k)_{k\in \N}$ is bounded,
    we have shown that the set $M$ of its limit points
    (defined as in \eqref{eq:27} with $\mathfrak T:=\N$)
    is included in $\Fix(T)$ and therefore it satisfies
    \eqref{eq:40} thanks to \eqref{eq:18}.
    An application of Lemma \ref{le:core} concludes the proof. 
\end{proof}
\end{thm}
When $\H$ has finite dimension we obviously get
narrow convergence in $\cP(\H)$.


\end{document}